\documentclass[11pt,bookmarks,colorlinks,colorlinks,citecolor=red,urlcolor=blue,linkcolor=blue]{epiarticle}
\usepackage{epimath}

\setpapertype{A4}

\usepackage[english]{babel}

\newtheorem{conjecture}[subsection]{Conjecture}

\newcommand{\mr}{\ensuremath{\mathbb R}}
\newcommand{\mc}{\ensuremath{\mathbb C}}
\newcommand{\dif}{\mathrm{d}}

\newcommand{\mq}{\ensuremath{\mathbb Q}}
\newcommand{\half}{\tfrac{1}{2}}

\newcommand{\mz}{\ensuremath{\mathbb Z}}
\newcommand{\leg}[2]{\left(\frac{#1}{#2}\right)}

\newcommand{\R}{\mathbb R}

\newcommand{\eps}{\varepsilon}

\title{One level density of low-lying zeros of quadratic Hecke $L$-functions to prime moduli}
\titlemark{\it One level density of low-lying zeros of quadratic Hecke $L$-functions to prime modulus}
\author{Peng Gao and Liangyi Zhao}
\authormark{P.~Gao and L.~Zhao}
\date{} 
\journal{Hardy-Ramanujan Journal 43 (2020), 173-187} 
\acceptation{submitted 15/05/2020, accepted 22/03/2021, revised 22/03/2021}

\begin{document}

\maketitle

\vspace{-0.3cm}


\thanks{The first named author is supported in part by NSFC grant 11871082, and the second named author by the FRG grant PS43707 and the Goldstar Award PS53450 at the University of New South Wales (UNSW).\\

\noindent
We thank \href{http://episciences.org}{episciences.org} for providing open access hosting of the electronic journal \emph{Hardy-Ramanujan Journal}}

\begin{prelims}

\vspace{-0.2cm}
\def\abstractname{Abstract}
\abstract{In this paper, we study the one level density of low-lying zeros of a family of quadratic Hecke $L$-functions to prime moduli over the Gaussian field under the generalized Riemann hypothesis (GRH) and the ratios conjecture. As a corollary, we deduce that at least $75 \%$ of the members of this family do not vanish at the central point under GRH.}

\vskip 0.1cm



\end{prelims}


\section{Introduction}

The low-lying zeros of families of $L$-functions have important applications in problems such as determining the size of the Mordell-Weil groups of elliptic curves and the size of class numbers of imaginary quadratic
number fields.  For this reason, much work (see \cite{FI, ILS, DuMi2, HuMi, RiRo, Royer, SBLZ1, HB1, Brumer, SJM, Young, Gu2, DM, HuRu}) has been done
towards  the density conjecture of N. Katz and P. Sarnak \cite{KS1, K&S}, which relates the distribution of zeros near the central point of a family of
$L$-functions to the eigenvalues near $1$ of a corresponding classical compact group. \newline

   In this paper, we focus on $L$-functions attached to quadratic characters. For the family of quadratic Dirichlet $L$-functions, this is initiated by  A.
   E. \"{O}zl\"uk and C. Snyder in \cite{O&S}, who studied the 1-level density of low-lying zeros of the family. Subsequent investigations were carried out
   in  \cite{Gao, Ru, Miller1}, in which the cardinalities of families considered all have positive densities in the set of all such $L$-functions. For
   these families, the density conjecture is verified when the Fourier transforms of the test functions are supported in $(-2, 2)$ if one assumes the
   Generalized Riemann Hypothesis (GRH) and the underlying symmetry of such families is unitary symplectic (USp). \newline

   Recently, J. C. Andrade and S. Baluyot \cite{A&B} studied the $1$-level density of quadratic Dirichlet $L$-functions over prime moduli. This is a
   sparse family in the sense that its cardinality has density $0$ in the set of all such $L$-functions. It is shown in \cite{A&B} that the symmetry of
   this family is also USp and their result supports the density conjecture when the Fourier transforms of the test functions are supported in $(-1, 1)$ under
   GRH.\newline

   It is then interesting to study the $1$-level density of various families of $L$-functions of sparse sets. Motivated by the above result of Andrade and Baluyot, we investigate in this paper the $1$-level density of quadratic Hecke $L$-functions in the Gaussian field $\mq(i)$ over prime moduli.
   Previously, we studied in \cite{G&Zhao4} the same family but over a set of positive density in the set of all such $L$-functions. \newline

   Throughout the paper, we denote by $K=\mq(i)$ the Gaussian field and $\mathcal{O}_K=\mz[i]$ the ring of integers in $K$. Note that in
   $\mathcal{O}_K$, every ideal co-prime to $2$ has a unique generator congruent to $1$ modulo $(1+i)^3$.  Such a generator is called primary.  
 
 \newpage
 \setcounter{page}{174}
The quadratic symbol $\leg {\cdot}{\cdot}$ is defined in \cite[Sect. 2.1]{G&Zhao4} and we shall write $\chi_n$ for $\leg {n}{\cdot}$. It is
   also shown in \cite[Section 2.1]{G&Zhao4} that the symbol $\chi_{(1+i)^5c}$ defines a primitive quadratic Hecke character modulo $(1+i)^5c$ of trivial infinite type when $c \in \mathcal{O}_K$ is odd and square-free.  Here we recall that a Hecke character $\chi$ of $K$ is said to be of trivial infinite
   type if its component at infinite places of $K$ is trivial and we say that any $c \in \mathcal{O}_K$ is odd if $(c,2)=1$ and $c$ is square-free if the
   ideal $(c)$ is not divisible by the square of any prime ideal. \newline

   Throughout the paper, we reserve the symbol $\varpi$ for primes with $(\varpi, 1+i)=1$.  This means that $(\varpi)$ is a prime ideal in $\mathcal{O}_K$. We would like to consider the family of $L$-functions consisting of $L(s, \chi_{\varpi})$ for $\varpi$ being prime . Even though this is a natural choice,  we
   consider instead in this paper the family
\begin{align}
\label{F}
 \mathcal F = \big\{ L(s,\chi_{(1+i)^5\varpi}) : \varpi \text{ primary} \big\}
\end{align}
   since the modulus of $\chi_{(1+i)^5\varpi}$ is easier to describe. We remark here that our treatment for the above family certainly carries over to the family $\{ L(s, \chi_{\varpi}) \}$ as well. \newline

    Let $\chi=\chi_n$ for some $n \in \mathcal{O}_K$, we write $L(s, \chi)$ for the corresponding Hecke $L$-function and denote the non-trivial zeroes of $L(s,
    \chi)$ by $\half+i \gamma_{\chi, j}$.  Without assuming GRH (so that $\gamma_{\chi, j}$ is not necessarily real), we order them as
\begin{equation*}
    \ldots \leq
   \Re \gamma_{\chi, -2} \leq
   \Re \gamma_{\chi, -1} < 0 \leq \Re \gamma_{\chi, 1} \leq \Re \gamma_{\chi, 2} \leq
   \ldots.
\end{equation*}
    We further normalize the zeros by letting
\begin{align*}
    \tilde{\gamma}_{\chi, j}= \frac{\gamma_{\chi, j}}{2 \pi} \log X
\end{align*}
and define, for an even Schwartz class function $\phi$, the $1$-level density for the single $L$-function $L(s, \chi)$ as the sum
\begin{equation*}
S(\chi, \phi)=\sum_{j} \phi(\tilde{\gamma}_{\chi, j}).
\end{equation*}

    Fix a non-negative and not identically zero Schwartz function $w(t)$ which is compactly supported when restricted to $(0, \infty)$.  The $1$-level density of the family $\mathcal F$ is the limit, as $X \to \infty$, of the sum
\begin{align}
\label{Ddef}
  D(\phi;w, X) =\frac 1{W(X)}\sum_{\varpi \equiv 1 \bmod {(1+i)^3}}  w\left( \frac {N(\varpi)}X \right)  S(\chi_{(1+i)^5\varpi}, \phi),
\end{align}
    where we use the expression $\varpi \equiv 1 \bmod {(1+i)^3}$ to indicate that $\omega \in \mathcal{O}_K$ is primary and we denote by $N(n)$ the norm of any element $n \in \mathcal{O}_K$. Moreover $W(X)$ is the corresponding total weight given by
\begin{align}
\label{Wx}
  W(X)=\sum_{\varpi \equiv 1 \bmod {(1+i)^3}} w\left( \frac {N(\varpi)}X \right).
\end{align}

    We point out here that our formulation of the $1$-level density is the more commonly used one in the literature, while in \cite{A&B}, the $1$-level density is formed using a form factor, as initially used by \"{O}zl\"uk and Snyder in \cite{O&S}. Now we state our result on the one level density as follows.
\begin{theorem}
\label{quadraticmainthm}
Assume the truth of GRH for all the $L$-functions in $\mathcal{F}$.  Let $w(t)$ be a not identically zero and non-negative Schwartz function.  Suppose that when restricted to $(0,\infty)$, $w$ has compact support.  Further, let $\phi(x)$ be an even Schwartz function whose
Fourier transform $\hat{\phi}(u)$ has compact support in $(-1,1)$.  Then,
\begin{align}
\label{quaddensity}
 \lim_{X \rightarrow +\infty}D(\phi;w, X)
 = \int\limits_{\mathbb{R}} \phi(x)  W_{USp}(x) \dif x, \; \; \mbox{where} \; \; W_{USp}(x)=1-\frac{\sin(2\pi x)}{2\pi x}.
\end{align}
\end{theorem}

First we note that the compact support of $w$ on $(0, \infty)$ is required for technical ease (see, among other things, the discussion around \eqref{Wsplit}).  Moreover, this condition enables us to use Lemma~\ref{lem2.1} directly. \newline

We further remark here that the kernel of the integral $W_{USp}$ in \eqref{quaddensity} shows that the symmetry type of this family of
quadratic Hecke $L$-functions is unitary symplectic and Theorem \ref{quadraticmainthm} implies that under GRH, the density conjecture is valid for
 this family when the Fourier transforms of the test functions are supported in $(-1, 1)$.  The reason that our result only holds for $\hat{\phi}(u)$ being
 supported in $(-1,1)$ in contrast to $(-2,2)$ obtained in \cite{Gao, G&Zhao4} for families of quadratic $L$-functions having positive densities in the set
 of all such $L$-functions is that one is not able to apply Poisson summation to convert certain character sums to dual sums to obtain a better estimate.  We are therefore bound to deploy only classical methods to estimate these sums which leads to a smaller admissible range of support for $\hat{\phi}$.  Assuming the truth of GRH meanwhile helps us better
 control the error terms coming from certain sums over primes. \newline

The density conjecture can be regarded as an assertion on the main term behavior of the $n$-level density of low-lying zeros of families of $L$-functions for all $n$. Besides the main term, one can also study the lower order terms of these $n$-level densities and computations as such are done in \cite{Young20015, RR, Miller2009}. In this direction, the $L$-functions ratios conjecture of J. B. Conrey, D. W. Farmer and M. R. Zirnbauer in \cite[Section 5]{CFZ} has been shown to be a very useful tool.  For various families of quadratic Dirichlet $L$-functions, ratios conjecture has been applied to predict precise lower order terms of the one level density in \cite{CS,FPS,A&B}. \newline

  In Section \ref{The ratios conjecture's prediction}, we apply the ratios conjecture to investigate the lower-order terms of the one level density of low-lying zeros for the family $\mathcal{F}$ given in \eqref{F}. To state our result, we define for $\alpha, \gamma \in \mc$,
\begin{align} \label{defnofae}
A(\alpha,\gamma)= \frac{2^{1+\alpha+\gamma}-2^{\gamma-\alpha}}{2^{1+\alpha+\gamma}-1} \quad \mbox{and} \quad
A_{\alpha}(r,r) = \frac{\partial}{\partial \alpha}A(\alpha,\gamma)\bigg|_{\alpha=\gamma=r}.
\end{align}
  Then we have the following asymptotic expression for $D(\phi;w, X)$.
\begin{theorem}
\label{mainthmrc}
Assume the truth of GRH and Conjecture \ref{ratiosconjecture} (the ratios conjecture) for $\mathcal{F}$. Let $w(t)$ and $\phi(x)$ be as in Theorem~\ref{quadraticmainthm}.  Then for any $\eps>0$,
\begin{align} \label{quaddensityrc}
D(\phi;w, X)= \frac{1}{W(X)}  \sum_{\substack{ \varpi \equiv 1 \bmod {(1+i)^3}}}  w\left( \frac {N(\varpi)}X \right) \mathfrak{I} (\varpi)  + O_{\varepsilon}\left( X^{-1/2+\varepsilon} \right),
\end{align}
where
\begin{equation*}
\begin{split}
\mathfrak{I} (\varpi) =  \frac{1}{2\pi} \int\limits_{\mathbb R} \bigg(2  \frac{\zeta'_K(1+2it)}{\zeta_K(1+2it)} + 2A_{\alpha}(it,it) &+\log \left(\frac {32N(\varpi)}{\pi^2}\right) +\frac{\Gamma'}{\Gamma}\left(\frac 12-it\right) + \frac{\Gamma'}{\Gamma}\left(\frac 12+it  \right)  \\
& -\frac {8}{\pi}  X_{\varpi}\left( \frac{1}{2}+it\right)\zeta_K(1-2it)A(-it,it) \bigg) \, \phi\left(\frac{t\log X}{2\pi}\right) \, \dif t .
\end{split}
\end{equation*}
\end{theorem}

  One certainly expects that the result given in Theorem \ref{mainthmrc} implies the assertion of Theorem \ref{quadraticmainthm}. We show this indeed is the case in Section \ref{The ratios conjecture's prediction} by adapting the approach in \cite{FPS1}.

  One important application of the density conjecture is to obtain a non-vanishing result for the families of $L$-functions at the central point.  It is a conjecture that goes back to S. Chowla \cite{chow} that $L(1/2, \chi) \neq 0$ for all primitive Dirichlet characters $\chi$.  It is shown in \cite[Theorem 3]{A&B} that at least $75\%$ of the family of quadratic Dirichlet $L$-functions to prime moduli does not vanish at the central point. Analogous to this, the following corollary shows that exactly the same percentage of the family of quadratic Hecke $L$-functions to prime moduli does not vanish at the central point.

 \begin{corollary} \label{quadnonvan}
 Assume GRH and that $1/2$ is a zero of $L \left( s, \chi_{(1+i)^5\varpi} \right)$ of order $m_{\varpi} \geq 0$.  As $X \to \infty$,
\begin{align}
\label{zerobound}
 \sum_{\varpi \equiv 1 \bmod {(1+i)^3}} m_{\varpi} w \left( \frac {N(\varpi)}{X} \right) \leq  \left( \frac{1}{4} + o_X(1) \right) W(X).
\end{align}
   Moreover, as $X \to \infty$
\begin{align}
\label{nonvanishing} \# \{ N(\varpi) \leq X : L\left( 1/2, \chi_{(1+i)^5\varpi} \right) \neq 0 \} \geq \left( \frac{3}4+ o_X(1) \right) \frac {X}{\log X} .
\end{align}
\end{corollary}

   As the proof of \eqref{zerobound} is standard (see that of \cite[Corollary 2.1]{B&F}), we omit it here and we note that \eqref{zerobound} implies that
\begin{align*}
 \sum_{\substack{\varpi \equiv 1 \bmod {(1+i)^3} \\ L\left( 1/2, \chi_{(1+i)^5\varpi} \right) \neq 0 }} w \left( \frac {N(\varpi)}{X} \right) \geq  \left( \frac{3}{4} + o_X(1) \right) W(X).
\end{align*}
 By taking $w(t)$ to be any even Schwarz function whose support is in $(0,1)$ when restricted on the positive real axis such that $0 \leq w(t) \leq 1$ and $w(t)=1$ for $t \in (\varepsilon, 1-\varepsilon)$ and applying \eqref{W} below, we deduce from the above that
\begin{align*}
 \# \{ N(\varpi) \leq X : L\left( 1/2, \chi_{(1+i)^5\varpi} \right) \neq 0 \} \geq  \left( \frac{3}{4} + o_X(1) \right) (1-2\varepsilon)\frac {X}{\log X}.
\end{align*}
 Letting $\varepsilon \rightarrow 0^+$ on both sides above, we see that \eqref{nonvanishing} follows.

\section{Preliminaries}
\label{sec 2}

\subsection{The Explicit Formula}
\label{section: Explicit Formula}

   Our proof of Theorem \ref{quadraticmainthm} starts with the following explicit formula, which
converts a sum over zeros of an $L$-function to a sum over primes.  Note that our weight function $w$ is compactly supported on $(0,\infty)$.  Thus the following lemma is sufficient for our purpose.
\begin{lemma}
\label{lem2.1}
   Let $\phi(x)$ be an even Schwartz function whose Fourier transform
   $\hat{\phi}(u)$ is compactly supported. For any square-free $c \in \mathcal{O}_K, N(c) \leq X$,
\begin{align*}
  S(\chi_{c}, \phi)   =\hat{\phi}(0) \frac {\log N(c)}{\log X} -\frac 1{2}
   \int\limits^{\infty}_{-\infty}\hat{\phi}(u) \dif u-2 S(\chi_{c},X; \hat{\phi})+O \left( \frac {\log \log 3X}{\log
   X} \right),
\end{align*}
    with the implicit constant depending on $\phi$. Here
\begin{align*}
    S (\chi_c, X; \hat{\phi}) =\ \sum_{\substack{ \varpi \equiv 1 \bmod {(1+i)^3} \\ N(\varpi) \leq X}} \frac {\chi_{c}(\varpi)\log
    N(\varpi)}{\sqrt{N(\varpi)}}\hat{\phi} \left( \frac {\log N(\varpi)}{\log X} \right).
\end{align*}
\end{lemma}
   We omit the proof of Lemma \ref{lem2.1} here since it is standard and follows by combining the proof of \cite[Lemma 4.1]{Gu} and \cite[Lemma 2.4]{G&Zhao4}.

\subsection{Conditional Estimates on GRH}
\label{section: Conditional Estimates on GRH}

    In this section, we include two lemmas that are obtained by assuming the truth of GRH. The first is about sums over primes.
\begin{lemma}
\label{lem2.7}
Assume the truth of GRH for $L(s,\chi)$ with Hecke character $\chi \pmod {m}$ of trivial infinite type.  Then for $y \geq 1$,
\begin{align}
\label{lem2.7eq}
  S(y, \chi)=\sum_{\substack {\varpi \equiv 1 \bmod {(1+i)^3} \\ N(\varpi) \leq y }} \chi (\varpi) \log N(\varpi)
=\delta_{\chi} y+ O\big( y^{1/2} \log^{2} (2y) \log (2N(m)) \big ),
\end{align}
    where $\delta_{\chi}=1$ if and only if $\chi$ is principal and $\delta_{\chi}=0$ otherwise.  Moreover,
    \begin{align} \label{mer}
  \sum_{\substack{ \varpi \equiv 1 \bmod {(1+i)^3} \\ N(\varpi) \leq y }} \frac {\log N(\varpi)}{N(\varpi) }= \log y+O(1).
\end{align}
\end{lemma}
\begin{proof}
   The formula in \eqref{lem2.7eq} follows directly from  \cite[Theorem 5.15]{iwakow} and \eqref{mer} follows from \eqref{lem2.7eq} by taking $\chi$ to be the principal character and using partial summation.
\end{proof}

   We recall that the Mellin transform of $w$ is given by
\begin{align}
\label{wMellin}
     \widetilde{w}(s) =\int\limits^{\infty}_0w(t)t^s\frac {\dif t}{t}.
\end{align}
   Integrating by parts implies that for $\Re(s)<1$ and any integer $\nu \geq 1$,
\begin{align}
\label{equ:bd for fei}
 \widetilde{w}(s) \ll _{\nu} \frac{1}{|s||s-1|^{\nu-1}}.
\end{align}

Let $\zeta_K(s)$ denote the Dedekind zeta function of $K$ and $\Lambda_{K}(n)$ the von Mangoldt function on $K$, which is the coefficient of $N(n)^{-s}$ in the Dirichlet series of $\zeta^{'}_K(s)/\zeta_K(s)$. Our next result gives estimates of various sums needed in this paper.
\begin{lemma}
\label{lemma logd}
Assume GRH. For any even, not identically zero and non-negative Schwartz function $w$ which is compactly supported when restricted to $(0, \infty)$, let $W(X)$ be given as in \eqref{Wx} for $X \geq 1$.  Let $z \in \mc$ be such that $|z| \leq 1$ and that $0\leq \Re(z) \leq \frac 12$. Then for any $\varepsilon>0$,
\begin{equation} \label{wsum}
 \sum_{\varpi \equiv 1 \bmod {(1+i)^3}}  w\left( \frac {N(\varpi)}X \right)  \log N(\varpi)=  \widetilde{w}(1)X+O \left( X^{1/2+\varepsilon} \right),
 \end{equation}
 \begin{equation} \label{W}
 W(X) = \widetilde{w}(1)\frac {X}{\log X}+O\left(\frac {X}{(\log X)^2} \right ),
 \end{equation}
 and
 \begin{equation}
\label{Wpower}
\frac{1}{W(X)}\sum_{\varpi \equiv 1 \bmod {(1+i)^3}}  w\left( \frac {N(\varpi)}X \right) N(\varpi)^{-z} = X^{-z}  +O\left( |z|^2\log X+\frac {1}{\log X} \right).
\end{equation}
\end{lemma}
\begin{proof}
  We prove \eqref{wsum} first. Due to the
  rapid decay of $w$ given in \eqref{equ:bd for fei}, we have that
\begin{align*}
 \sum_{\varpi \equiv 1 \bmod {(1+i)^3}}\!\!\!\!\!\! w\!\left( \frac {N(\varpi)}X \right)  \log N(\varpi)=& \sum_{(n)}w\!\left( \frac {N(n)}X \right)\!\Lambda_{K}(n)
 +O\! \left(\sum_{\substack{(\varpi)\\ N(\varpi)^j \leq X^{1+\varepsilon}, \ j \geq 2 }}\!\!\!\!\!\!\!\! w\left( \frac {N(\varpi^j)}X \right)  \log N(\varpi) \right),
\end{align*}
  where we write $\sum_{(n)}$ and $\sum_{(\varpi)}$ for the sum over non-zero integral and prime ideals of $\mathcal{O}_K$, respectively. \newline

   Note that
\begin{align}
\label{higherpowerest}
 \sum_{\substack{(\varpi), \ j \geq 2 \\ N(\varpi)^j \leq X^{1+\varepsilon}}} w\left( \frac {N(\varpi^j)}X \right)  \log N(\varpi) \ll
 X^{\varepsilon}\sum_{\substack{(\varpi) \\ N(\varpi) \leq X^{1/2+\varepsilon}}} \log N(\varpi) \ll X^{1/2+\varepsilon},
\end{align}
where we have used the bound $w(u) \ll u^{-\varepsilon}$, as $w$ is a Schwartz function. \newline

   Now we apply Mellin inversion to get
\begin{align*}
\sum_{(n)} w\left( \frac {N(n)}X  \right)  \Lambda_{K}(n)
=& -\frac 1{2\pi i}\int\limits_{(2)} \frac {\zeta_K'(s)}{\zeta_K(s)} \widetilde{w}(s)X^s \dif s.
\end{align*}

  We evaluate the above integral by shifting the line of integration to $\Re(s)=1/2+\varepsilon$. The only pole we encounter is at $s=1$ with residue $-\widetilde{w}(1)X$. The integration on $\Re(s)=1/2+\varepsilon$ can be estimated to be $O(X^{1/2+\varepsilon})$ using \eqref{equ:bd for fei} and the following estimate (see \cite[Theorem 5.17]{iwakow}) for $\zeta_K'(s)/\zeta_K(s)$ when $\Re(s) \geq 1/2+\varepsilon$:
\begin{align*}
  \frac {\zeta_K'(s)}{\zeta_K(s)} \ll \log (1+|s|).
\end{align*}
  The expression given in \eqref{wsum} now follows. \newline

  Next, to obtain \eqref{W}, we may assume that $X$ is large and apply \eqref{lem2.7eq} and partial summation to see that
\begin{align}
\label{Wsplit}
\begin{split}
 W(X) =&   \int\limits^{\infty}_{2^-} w\left( \frac uX \right) \frac 1{\log u} \dif \Big ( \sum_{\substack{\varpi \equiv 1 \bmod {(1+i)^3} \\ N(\varpi) \leq u}} \log N(\varpi) \Big )=\int\limits^{\infty}_{2^-} w\left( \frac uX \right) \frac 1{\log u} \dif \Big (u+O \big (u^{1/2} \log^{2} (2u) \big ) \Big ) \\
 =& \int\limits^{\infty}_{2^-} w\left( \frac uX \right) \frac 1{\log u} \dif u+O(X^{1/2+\varepsilon})  = \int\limits^{\infty}_{(2/X)^-} w\left(  u \right) \frac X{\log u+\log X} \dif u+O(X^{1/2+\varepsilon})  \\
 =  & \int\limits_m^M w\left(  u \right) \frac X{\log u+\log X} \dif u+O(X^{1/2+\varepsilon}) ,
\end{split}
\end{align}
where $[m,M] \subset (0,\infty)$ is the support of $w$.  Note that $m$ and $M$ are independent of $X$.  Now we use the Taylor expansion of $1/(\log u + \log X)$ and then extend the range of integration to $(0,\infty)$ to get
\[ W(X) = \widetilde{w}(1)\frac {X}{\log X}+O\left(\frac {X}{(\log X)^2} \right ) .\]

  It therefore remains to establish \eqref{Wpower}. For this, we set
\begin{equation*}
 f(z)=\sum_{\varpi \equiv 1 \bmod {(1+i)^3}}  w\left( \frac {N(\varpi)}X \right) N(\varpi)^{-z}.
\end{equation*}
  Then we have
\begin{align*}
 f'(z)=-\!\!\!\!\sum_{\varpi \equiv 1 \bmod {(1+i)^3}} \!\!\!\! w\!\left( \frac {N(\varpi)}X \right)( \log N(\varpi)) N(\varpi)^{-z} =-\sum_{\substack{ (n)}}  w\!\left( \frac {N(n)}X \right) \Lambda_{K}(n) N(n)^{-z}+O(X^{1/2+\varepsilon}),
\end{align*}
  where the last estimation above follows from \eqref{higherpowerest}, uniformly for $z$ with $\Re z \geq 0$. \newline

   Now we apply Mellin inversion to get
\begin{align*}
\sum_{(n)} w\left( \frac {N(n)}X  \right)  \Lambda_{K}(n)N(n)^{-z}
=& -\frac 1{2\pi i}\int\limits_{(2)} \frac {\zeta_K'(s+z)}{\zeta_K(s+z)} \widetilde{w}(s)X^s \dif s.
\end{align*}

  We evaluate the above integral by shifting the line of integration to $\Re(s)=1/2-\Re(z)+\varepsilon$. The only pole we encounter is at $s=1-z$ with residue $\widetilde{w}(1-z)X^{1-z}$. Thus we obtain that
\begin{align*}
f'(z)= \widetilde{w}(1-z)X^{1-z}+O\left( X^{1/2+\varepsilon} \right).
\end{align*}
  Note that we have
\begin{align*}
 \widetilde{w}(1-z)-\widetilde{w}(1)=\int\limits^{-z}_0\widetilde{w}'(1+s) \dif s,
\end{align*}
  where the integral takes the path being the line segment connecting $0$ and $-z$.  Based on our conditions on $z$ and the definition of $\widetilde{w}$, one sees that the function $\widetilde{w}'(1+s)$ is uniformly bounded on the path so that we deduce that $\tilde{w}(1-z) = \tilde{w}(1) + O(|z|)$.  As $X^{1-z} \ll X$ for $\Re z \geq 0$, we conclude that
\begin{align*}
f'(z)=\widetilde{w}(1)X^{1-z}+O \left( |z|X+X^{1/2+\varepsilon} \right).
\end{align*}
 It follows from this that
\begin{align*}
  f(z)-f(0)= \int\limits^z_0f'(v) \dif v = \widetilde{w}(1)\frac {X^{1-z}}{\log X}-\widetilde{w}(1)\frac {X}{\log X}+O \left( |z|^2X+X^{1/2+\varepsilon} \right).
\end{align*}

  Note that we have $f(0)=W(X)$ so that we deduce from the above and \eqref{W} that
\begin{align*}
  f(z) =\widetilde{w}(1)\frac {X^{1-z}}{\log X}+O \left( |z|^2X+\frac {X}{(\log X)^2}+X^{1/2+\varepsilon} \right).
\end{align*}
   Combining this with \eqref{W} again for the evaluation of $W(X)$, we readily deduce \eqref{Wpower} and this completes the proof of the lemma.
\end{proof}

\subsection{The approximate functional equation}

   Let  $\chi$ be a primitive quadratic Hecke character modulo $m$ of trivial infinite type defined on $\mathcal{O}_K$.
As shown by E. Hecke, $L(s, \chi)$ admits
analytic continuation to an entire function and satisfies the
functional equation (\cite[Theorem 3.8]{iwakow})
\begin{align}
\label{1.1}
  \Lambda(s, \chi) = W(\chi)(N(m))^{-1/2}\Lambda(1-s, \chi),
\end{align}
   where $|W(\chi)|=(N(m))^{1/2}$ and
\begin{align*}
  \Lambda(s, \chi) = (|D_K|N(m))^{s/2}(2\pi)^{-s}\Gamma(s)L(s, \chi).
\end{align*}

  Let $G(s)$ be any even function which is holomorphic and bounded in the strip $-4<\Re(s)<4$ satisfying $G(0)=1$. For $s \in \mc$, we evaluate the
    integral
\begin{equation*}
   \frac {1}{2\pi
   i}\int\limits\limits_{(2)}\Lambda(u+s, \chi)G(u) \Big (\frac {x}{\sqrt{|D_K|N(m)}} \Big )^u \frac{\dif u}{u}
\end{equation*}
   by  shifting the line of integration to $\Re s = -2$ and proceed in a manner similar to \cite[Section 2.5]{Gao1} to see that if
\begin{align} \label{Wchi}
  W(\chi)=N(m)^{1/2},
\end{align}
   then for any $x>1$,
\begin{align} \label{quadapproxfuneqQi}
\begin{split}
 L(s, \chi)
   = \sum_{0 \neq \mathcal{A} \subset
  \mathcal{O}_K} & \frac{\chi(\mathcal{A})}{N(\mathcal{A})^{s}}V_s \left(\frac{2\pi  N(\mathcal{A})}{x} \right)  \\
  & + \left(\frac {(2\pi)^2}{|D_K|N(m)} \right
  )^{s-1/2} \frac {\Gamma (1-s)}{\Gamma (s)}\sum_{0 \neq \mathcal{A} \subset
  \mathcal{O}_K}\frac{\chi(\mathcal{A})}{N(\mathcal{A})^{1-s}}V_{1-s}\left(\frac{2\pi
  N(\mathcal{A})x}{|D_K|N(m)} \right).
\end{split}
\end{align}
    where
\begin{align}
\label{2.14}
  V_{s} \left(x \right)=\frac {1}{2\pi
   i}\int\limits\limits_{(2)}\frac {\Gamma(s+u)}{\Gamma (s)} \frac
   {x^{-u}}{u} G(u) \ \dif u.
\end{align}

   We note that it is shown in \cite[Lemma 2.2]{Gao2} that \eqref{Wchi} holds if $\chi=\chi_{(1+i)^5c}$ for any odd, square-free $c \in
   \mathcal{O}_K$. Thus the approximate functional equation \eqref{quadapproxfuneqQi} is valid for $L(s, \chi_{(1+i)^5c})$.
\section{Proof of Theorem \ref{quadraticmainthm}}
\label{Section 3}

   Note that $\hat{\phi}(u)$ is smooth with support contained in $(-1+\varepsilon, 1-\varepsilon)$ for some $0<\varepsilon<1$. We set $Y=X^{1-\varepsilon}$
   so that $\hat{\phi} \left( \log N(\varpi)/\log X \right) \neq 0$ only when $N(\varpi) \leq Y$. Now we apply Lemma \ref{lem2.1} to sum $S(\chi_{(1+i)^5\varpi}, X;
   \hat{\phi})$ over the primary primes $\varpi$ against the weight function $w$ to arrive at
\begin{align}
\label{D}
\begin{split}
    D(\phi;w, X) =&  \frac {\hat{\phi}(0)}{W(X)\log X}\sum_{\substack{ \varpi \equiv 1 \bmod {(1+i)^3}}} w\left( \frac {N(\varpi)}X \right)  \log N(\varpi) \\
    & \hspace*{3cm} -\frac 1{2}
   \int\limits^{\infty}_{-\infty}\hat{\phi}(u) \dif u-\frac {2}{W(X)}S(X, Y; \phi, w)+O\left(\frac {\log \log 3X}{\log X} \right )  \\
   =& \int\limits^{\infty}_{-\infty}\phi(t) \dif t-\frac 1{2}
   \int\limits^{\infty}_{-\infty}\hat{\phi}(u) \dif u-\frac {2}{W(X)}S(X, Y; \phi, w)+O\left(\frac {\log \log 3X}{\log X} \right ),
\end{split}
\end{align}
    where the last estimate follows from \eqref{wsum}, \eqref{W} and
\begin{align*}
    S(X, Y; \phi, w)=\sum_{\substack{ \varpi \equiv 1 \bmod {(1+i)^3}}} \ \sum_{\substack{\varpi' \equiv 1 \bmod {(1+i)^3} \\ N(\varpi') \leq Y }} \frac {\chi_{(1+i)^5\varpi}(\varpi')\log N(\varpi')}{\sqrt{N(\varpi')}}\hat{\phi} \left(
    \frac {\log N(\varpi')}{\log X} \right) w \left( \frac {N(\varpi)}{X} \right).
\end{align*}

    We note that by the quadratic reciprocity (\cite[(2.1)]{G&Zhao4}) for Hecke characters, we have $\chi_{\varpi}(\varpi')=\chi_{\varpi'}(\varpi)$ when $\varpi, \varpi'$
    are both primary. As it is shown in \cite[Sect. 2.1]{G&Zhao2019} that $\chi_{\varpi'}$ is a Hecke character modulo $16\varpi'$ of trivial infinite type,
    we can apply \eqref{lem2.7eq} and partial summation by noting that $\log N(\varpi') \ll \log X$ to see that
\begin{equation*}
  \sum_{\substack{ \varpi \equiv 1 \bmod {(1+i)^3}}}\chi_{\varpi}(\varpi') w \left( \frac {N(\varpi)}{X} \right) \ll X^{1/2+\varepsilon/4}.
\end{equation*}

    It follows that
\begin{align*}
    S(X,Y; \phi, w) =& \sum_{\substack{\varpi' \equiv 1 \bmod {(1+i)^3} \\ N(\varpi') \leq Y }} \frac {\log N(\varpi')\chi_{(1+i)^5}(\varpi')}{\sqrt{N(\varpi')}}\hat{\phi} \left( \frac {\log
    N(\varpi')}{\log X} \right)\sum_{\substack{ \varpi \equiv 1 \bmod {(1+i)^3}}} \chi_{\varpi}(\varpi') w \left( \frac {N(\varpi)}{X} \right) \\
    \ll & X^{1/2+\varepsilon/4} \sum_{\substack{\varpi' \equiv 1 \bmod {(1+i)^3} \\ N(\varpi') \leq Y }} \frac {\log N(\varpi')}{\sqrt{N(\varpi')}} \ll X^{1/2+\varepsilon/4}Y^{1/2} \log Y,
\end{align*}
  where the last bound follows from \eqref{mer} and partial summation. \newline

    We then deduce that when $Y=X^{1-\varepsilon}$, $S(X,Y; \phi, w)=O(X^{1-\varepsilon})=o(W(X))$ by Lemma \ref{lemma logd}.  By taking $X \rightarrow \infty$ on both sides of \eqref{D}, we obtain \eqref{quaddensity} and this completes the proof of Theorem \ref{quadraticmainthm}.

\section{One level density via ratios conjecture}
\label{The ratios conjecture's prediction}

\subsection{The Ratios Conjecture for $\mathcal F$}
	
 In this section, we follow the recipe described in \cite{CFZ} to develop heuristically the appropriate statement of the ratios conjecture for the family $\mathcal F$ given in \eqref{F}. More precisely, we want to study the asymptotic behavior of
\begin{align}
\label{ralphgam}
R(\alpha,\gamma)=\frac1{W(X)}\sum_{\substack{ \varpi \equiv 1 \bmod {(1+i)^3}}}  w\left( \frac {N(\varpi)}X \right)
\frac{L\left(1/2+\alpha,\chi_{(1+i)^5\varpi}\right)}{L\left(1/2+\gamma,\chi_{(1+i)^5\varpi}\right)}.
\end{align}
	
Consider the approximate functional equation \eqref{quadapproxfuneqQi}.  Note that the function $V_s(t)$ given in \eqref{2.14} is essentially $1$ when $|t|\leq 1$ and decreases exponentially when $|t|>1$ for fixed $s$ satisfying $\Re(s) \geq 1/4$ (this can be shown similar to \cite[Lemma 2.1]{sound1}).  Thus it stands to reason that for any odd, square-free $c \in \mathcal{O}_K$,
\begin{equation}
\label{approxeq}
L\left(s, \chi_{(1+i)^5c}\right) \approx \sum_{N(\mathfrak{n}) \leq x} \frac{\chi_{(1+i)^5c}(\mathfrak{n})}{N(\mathfrak{n})^s}+X_{c}(s)\sum_{N(\mathfrak{n})
\leq y}
\frac{\chi_{(1+i)^5c}(\mathfrak{n} )}{N(\mathfrak{n})^{1-s}},
\end{equation}
 where $xy=N((1+i)^5c)$ and
\begin{equation}
\label{xe}
X_{c}(s)=\frac{\Gamma\left(1-s\right)}{\Gamma\left(s \right)}\left(\frac{\pi^2}{32N(c)}\right)^{s-1/2}.
\end{equation}

 On the other hand, by writing $\mu_K$ for the M\"obius function on $K$, we have for $\Re(s)>1$,
\begin{equation}
\label{lemobius}
\frac{1}{L(s,\chi_{(1+i)^5c})} = \sum_{\mathfrak{m} \neq 0}
\frac{\mu_K(\mathfrak{m})\chi_{(1+i)^5c}(\mathfrak{m})}{N(\mathfrak{m})^s},
\end{equation}
  where the summation above is over all non-zero ideals $\mathfrak{m}$ of $\mathcal{O}_K$. \newline

 Substituting both \eqref{approxeq} and \eqref{lemobius} into \eqref{ralphgam}, we see that
\begin{equation*}
R(\alpha,\gamma) \sim R_1(\alpha, \gamma)+R_2(\alpha, \gamma),
\end{equation*}
 where
\[ R_1(\alpha, \gamma)= \frac1{W(X)}\sum_{\substack{ \varpi \equiv 1 \bmod {(1+i)^3}}} w\left( \frac {N(\varpi)}X \right) \sum_{\mathfrak{m},\mathfrak{n}}
\frac{\mu_K(\mathfrak{m})\chi_{(1+i)^5\varpi}(\mathfrak{nm})}{N(\mathfrak{m})^{1/2+\gamma}N(\mathfrak{n})^{1/2 +\alpha}}, \]
and
\[ R_2(\alpha, \gamma)=\frac1{W(X)}\sum_{\substack{ \varpi \equiv 1 \bmod {(1+i)^3}}}  w\left( \frac {N(\varpi)}X \right)X_{\varpi}\left( \frac{1}{2} + \alpha\right)
\sum_{\mathfrak{m},\mathfrak{n}}
\frac{\mu_K(\mathfrak{m})\chi_{(1+i)^5\varpi}(\mathfrak{nm})}{N(\mathfrak{m})^{1/2+\gamma}N(\mathfrak{n})^{1/2 -\alpha}}. \]

  We expect to get the main contributions to $R_1, R_2$ when $\mathfrak{nm}$ is an odd square.  In that case, we have
\begin{align*}
\frac{1}{W(X)}\sum_{\substack{ \varpi \equiv 1 \bmod {(1+i)^3}}} w\left( \frac {N(\varpi)}X \right) \chi_{(1+i)^5\varpi}(\mathfrak{nm}) \sim 1.
\end{align*}
 It follows that
\begin{equation*}
 R_1(\alpha, \gamma) \sim \widetilde R_1(\alpha, \gamma) =
\sum_{\mathfrak{nm} = \text{odd } \square} \frac{\mu_K(\mathfrak{m})}{N(\mathfrak{m})^{1/2+\gamma}N(\mathfrak{n})^{1/2 + \alpha}},
\end{equation*}
 where we write $\square$ for a perfect square. We deduce by computing Euler products that
\begin{align*}
\widetilde R_1(\alpha,\gamma)= \frac{\zeta_K(1+2\alpha)}{\zeta_K(1+\alpha+\gamma)}A(\alpha,\gamma),
\end{align*}
where $A(\alpha,\gamma)$ is given in \eqref{defnofae}. \newline

 Similarly, we have
\begin{equation*}
R_2(\alpha, \gamma) \sim \widetilde R_2(\alpha, \gamma)=\frac{1}{W(X)}\sum_{\substack{ \varpi \equiv 1 \bmod {(1+i)^3}}}  w\left( \frac {N(\varpi)}X \right)X_{\varpi}\left(\frac{1}{2} +
\alpha\right)\widetilde R_1(-\alpha,\gamma).
\end{equation*}

  We now summarize our discussions above in the following version of the ratios conjecture. 	
\begin{conjecture}
\label{ratiosconjecture}
Let $\varepsilon>0$ and let $w$ be an even and nonnegative Schwartz test function on $\R$ which is not identically zero. Suppose that the
complex numbers $\alpha$ and $\gamma$ satisfy $|\Re(\alpha)|< 1/4$, $(\log X)^{-1} \ll\Re(\gamma)<1/4$ and $\Im(\alpha)$, $\Im(\gamma) \ll
X^{1-\varepsilon}$. Then we have that
\begin{align} \label{Lratio}
\begin{split}
 \frac{1}{W(X)}\sum_{\substack{ \varpi \equiv 1 \bmod {(1+i)^3}}} & w\left( \frac {N(\varpi)}X \right) \frac{L(1/2 +\alpha, \chi_{(1+i)^5\varpi})}{L(1/2 + \gamma,\chi_{(1+i)^5\varpi})} =\frac{\zeta_K(1+2\alpha)}{\zeta_K(1+\alpha+\gamma)}A(\alpha,\gamma)\\
& +\frac{1}{W(X)}\sum_{\substack{ \varpi \equiv 1 \bmod {(1+i)^3}}}  w\left( \frac {N(\varpi)}X \right)
X_{\varpi}\left(\frac{1}{2}+\alpha\right) \frac{\zeta_K(1-2\alpha)}{\zeta_K(1-\alpha+\gamma)}A(-\alpha,\gamma)\\& +
O_\varepsilon\big(X^{-1/2+\varepsilon}\big),
\end{split}
\end{align}
where $A(\alpha,\gamma)$ is defined in $\eqref{defnofae}$ and $X_{\varpi}(s)$ is defined in $\eqref{xe}$.
\end{conjecture}
	
  By taking derivatives with respect to $\alpha$ on both sides of \eqref{Lratio} and noting that the residue of the simple pole of $\zeta_K(s)$ at $s=1$
  equals $\pi/4$, we deduce from Conjecture \ref{ratiosconjecture} the following result.
\begin{lemma} \label{ratiostheorem}
Assuming the truth of Conjecture \ref{ratiosconjecture}, we have for any $\varepsilon>0$, $(\log X)^{-1} \ll\Re(r)< 1/4$ and $\Im(r)\ll
X^{1-\varepsilon}$,
\begin{align}
\label{ratiostheoremthree}
\begin{split}
 \frac{1}{W(X)}\sum_{\substack{ \varpi \equiv 1 \bmod {(1+i)^3}}}  & w\left( \frac {N(\varpi)}X \right) \frac{L'(1/2 + r,\chi_{(1+i)^5\varpi})}{L(1/2 + r,\chi_{(1+i)^5\varpi})} =  \frac{\zeta'_K(1+2r)}{\zeta_K(1+2r)}+A_{\alpha}(r,r) \\
 & \!\!\!\!\!\!\!-\frac 4{\pi}\frac{1}{W(X)}\!\sum_{\substack{ \varpi \equiv 1 \bmod {(1+i)^3}}} \!\!\!\!\!\! w\!\left( \frac {N(\varpi)}X \right)\!
X_{\varpi}\!\left(\frac{1}{2}+r\right) \!\zeta_K(1-2r)A(-r,r)+ O_{\varepsilon}\big(X^{-1/2+\varepsilon}\big).
\end{split}
\end{align}
\end{lemma}

\subsection{Derivation of the one level density from the ratios conjecture}
\label{1levelsection}
	
In this section we prove Theorem \ref{mainthmrc} using the ratios conjecture and GRH. We recall the definition of  $D(\phi;w, X)$ from \eqref{Ddef} and we note that GRH implies that the non-trivial zeros of the $L$-functions all have real parts equal to $1/2$. Thus by setting $\mathcal{L}=\log X$, we can recast $D(\phi;w, X)$ as being equal to
\begin{align}
\label{cauchydens}
\begin{split}
=& \frac{1}{W(X)}\!\sum_{\substack{ \varpi \equiv 1 \bmod {(1+i)^3}}} \!\!\!\! w\left( \frac {N(\varpi)}X \right) \frac{1}{2\pi i}\left( \ \int\limits_{(a)} -
\int\limits_{(1-a)}\right)\!\frac{L'(s,\chi_{(1+i)^5\varpi})}{L(s,\chi_{(1+i)^5\varpi})}
\phi\left(\frac{-i\mathcal{L}}{2\pi}\left(s-\frac{1}{2}\right)\right) \dif s \\
=& \frac{1}{W(X)}\!\sum_{\substack{ \varpi \equiv 1 \bmod {(1+i)^3}}} \!\!\!\! w\!\left( \frac {N(\varpi)}X \right) \!\frac{1}{2\pi
i}\int\limits_{(a)}\!\!\left(\frac{L'(s,\chi_{(1+i)^5\varpi})}{L(s,\chi_{(1+i)^5\varpi})}-\frac{L'(1-s,\chi_{(1+i)^5\varpi})}{L(1-s,\chi_{(1+i)^5\varpi})}\right)\!
\phi\!\left(\frac{-i\mathcal{L}}{2\pi}\!\left(\!s-\frac{1}{2}\right)\!\right) \dif s ,
\end{split}
\end{align}
where $1/2+1/\log X <a< 3/4$ and the fact that $\phi$ is even is used in the derivation of the late equality above. \newline

 We note that the functional equation \eqref{1.1} implies that
\begin{equation*}
\frac{L'(s,\chi_{(1+i)^5\varpi})}{L(s,\chi_{(1+i)^5\varpi})} = \frac{X_{\varpi}'(s)}{X_{\varpi}(s)} -
\frac{L'(1-s,\chi_{(1+i)^5\varpi})}{L(1-s,\chi_{(1+i)^5\varpi})}.
\end{equation*}
Inserting the above into \eqref{cauchydens}, we obtain that $D(\phi;w, X)$ is equal to
\begin{align}
= \frac{1}{W(X)}\sum_{\substack{ \varpi \equiv 1 \bmod {(1+i)^3}}} \!\!w\left( \frac {N(\varpi)}X \right)  \frac{1}{2\pi i}\int\limits_{(a-1/2)} \!\!\!\left( 2 \frac{L'\left(1/2 +
r,\chi_{(1+i)^5\varpi}\right)}{L\left(1/2 + r,\chi_{(1+i)^5\varpi}\right)} -
\frac{X_{\varpi}'\left(1/2+r\right)}{X_{\varpi}\left(1/2+r\right)} \right) \phi\left(\frac{i\mathcal{L}r}{2\pi}\right) \dif r.
\label{checknine}
\end{align}
		
 Substituting \eqref{ratiostheoremthree} in \eqref{checknine}, we deduce by noting the rapid decay of $\phi$ that
\begin{align}
\label{RATIOSONELEV}
\begin{split}
D(\phi;w, X)= \frac{1}{W(X)}\sum_{\substack{ \varpi \equiv 1 \bmod {(1+i)^3}}} & w\left( \frac {N(\varpi)}X \right) \frac{1}{2\pi i}\int\limits_{(a-\tfrac 12)}  \bigg(2
\frac{\zeta'_K(1+2r)}{\zeta_K(1+2r)}+2A_{\alpha}(r,r)-\frac{X_{\varpi}'\left(1/2+r\right)}{X_{\varpi}\left(1/2+r\right)} \\
&\!\!\!\!\!\!\!\!\!\! -\frac {8}{\pi} X_{\varpi}\left(1/2+r\right)\zeta_K(1-2r)A(-r,r)\bigg) \phi\left(\frac{i\mathcal{L}r}{2\pi}\right)\dif r + O_{\varepsilon}\left( X^{-1/2+\varepsilon} \right).
\end{split}
\end{align}

  As the function
\begin{align*}
\begin{split}
2
\frac{\zeta'_K(1+2r)}{\zeta_K(1+2r)}+2A_{\alpha}(r,r)-\frac{X_{\varpi}'\left(1/2+r\right)}{X_{\varpi}\left(1/2+r\right)}  -\frac {8}{\pi} X_{\varpi}\left(\frac{1}{2}+r\right)\zeta_K(1-2r)A(-r,r)
\end{split}
\end{align*}
 is analytic in the region $\Re(r) \geq 0$ (in particular it is analytic at $r=0$), we can now shift the line of integration in \eqref{RATIOSONELEV} to  $\Re(r)=0$ to deduce readily the assertion of Theorem \ref{mainthmrc}.

\subsection{Proof of Theorem \ref{quadraticmainthm} using the ratios conjecture}
\label{section compare results}

 In this section, we give another proof of Theorem \ref{quadraticmainthm} by assuming the ratios conjecture. To achieve this, we recall that we set $\mathcal{L}=\log X$ and we first apply Lemma \ref{lemma logd} to see that
\begin{align}
\label{sectermest}
\frac{1}{W(X)}\sum_{\substack{ \varpi \equiv 1 \bmod {(1+i)^3}}}  w\left( \frac {N(\varpi)}X \right)  \frac{1}{2\pi}
\int\limits_{\mathbb R} \log\left(\frac{32N(\varpi)}{\pi^2}\right) \phi\left(\frac{t\mathcal{L}}{2\pi}\right) \, \dif t=\widehat \phi(0)+O\left( \frac {1}{\mathcal{L}} \right).
\end{align}
  Moreover, by \cite[Lemma 12.14]{MVa1}, we have that
\begin{align}
\label{thirdermest}
\begin{split}
\frac{1}{W(X)}\sum_{\substack{ \varpi \equiv 1 \bmod {(1+i)^3}}}  w\left( \frac {N(\varpi)}X \right) & \frac{1}{2\pi}
\int\limits_{\mathbb R}\left( \frac{\Gamma'}{\Gamma}\left(\frac 12-it\right)
+\frac{\Gamma'}{\Gamma}\left(\frac 12+it  \right)\right ) \phi\left(\frac{t\mathcal{L}}{2\pi}\right) \, \dif t\\
=& 2\frac{\Gamma'}{\Gamma}\left(\frac 12\right)\frac{\widehat{\phi}(0)}{\mathcal{L}} +\frac 2{\mathcal{L}}
\int\limits_0^\infty\frac{e^{-t/2}}{1-e^{-t}}\left(\widehat{\phi}(0)-\widehat{\phi}\left(\frac{t}{\mathcal{L}}\right)\right) \dif t=O\left( \frac {1}{\mathcal{L}}\right).
\end{split}
\end{align}

   We then deduce from \eqref{quaddensityrc}, \eqref{sectermest}, \eqref{thirdermest} that
\begin{align*}
\begin{split}
D(\phi;w, X)= \widehat \phi(0)+\frac{1}{W(X)}\sum_{\substack{ \varpi \equiv 1 \bmod {(1+i)^3}}} & w\left( \frac {N(\varpi)}X \right)  \frac{1}{2\pi}
\int\limits_{\mathbb R} \bigg(2 \frac{\zeta'_K(1+2it)}{\zeta_K(1+2it)} + 2A_{\alpha}(it,it)\\
&\!\!\!\!\!\!\!\!\!\!\!-\frac {8}{\pi}  X_{\varpi}\left(\frac{1}{2}+it\right)\zeta_K(1-2it)A(-it,it) \bigg) \,
\phi\left(\frac{t\mathcal{L}}{2\pi}\right) \, \dif t+ O\left( \frac {1}{\mathcal{L}}\right).
\end{split}
\end{align*}

  In view of \eqref{RATIOSONELEV} and the above expression for $D(\phi;w, X)$, we see that we can recast it as
\begin{align}
\label{Dsimplified}
\begin{split}
D(\phi;w, X)= \widehat \phi(0)+\frac{1}{W(X)}\sum_{\substack{ \varpi \equiv 1 \bmod {(1+i)^3}}} & w\left( \frac {N(\varpi)}X \right) \frac{1}{2\pi i}\int\limits_{(a')}  \bigg(2
\frac{\zeta'_K(1+2r)}{\zeta_K(1+2r)}+2A_{\alpha}(r,r)\\
&\!\!\!\!\!\!\!\!\!\!\!-\frac {8}{\pi} X_{\varpi}\left(\frac{1}{2}+r\right)\zeta_K(1-2r)A(-r,r)\bigg) \phi\left(\frac{i\mathcal{L}r}{2\pi}\right) \dif r + O\left( \frac {1}{\mathcal{L}} \right),
\end{split}
\end{align}
 where $1/\log X<a' <1/4$. By a straightforward computation we see that for $1/\log X<\Re(r) <1/4$, we have
\begin{equation*}
 A_{\alpha}(r, r)+\frac{\zeta'_K(1+2r)}{\zeta_K(1+2r)}=-\sum_{(\varpi)} \frac{\log N(\varpi)}{N(\varpi)^{1+2r}-1}.
\end{equation*}
 It follows from this and treatment similar to \cite[Lemma 4.1]{FPS1} and \cite[Lemma 3.7]{FPS} that we have
\begin{equation}
\label{equation lemma 4.1}
\frac{1}{2\pi i}\!\!\int\limits_{(a')} \!\!\! \bigg(2 \frac{\zeta'_K(1+2r)}{\zeta_K(1+2r)}+2A_{\alpha}(r,r)\bigg) \phi\!\left(\frac{i\mathcal{L}r}{2\pi}\right) \dif r =- \frac 2{\mathcal{L}
}\sum_{\substack{(\varpi) \\ j\geq 1}} \frac{\log N(\varpi)}{N(\varpi)^{j}} \widehat \phi\left( \frac{2j \log N(\varpi)}{\mathcal{L}} \!\right)=-\frac
{\phi(0)}{2}+O\left( \frac {1}{\mathcal{L}} \right)\!.
\end{equation}

  It now remains to treat the expression
\begin{align*}
\begin{split}
  I= -\frac {8}{\pi}\frac{1}{W(X)}\sum_{\substack{ \varpi \equiv 1 \bmod {(1+i)^3}}}  w\left( \frac {N(\varpi)}X \right) \frac{1}{2\pi i}\int\limits_{(a')}   X_{\varpi}\left(\frac{1}{2}+r\right)\zeta_K(1-2r)A(-r,r) \phi\left(\frac{i\mathcal{L}r}{2\pi}\right) \dif r .
\end{split}
\end{align*}

  Note that it follows from \eqref{defnofae} that
\begin{align*}
A(-\gamma,\gamma)=2-2^{2r}.
\end{align*}

  Combining this with the definition of $X_{\varpi}$ given in \eqref{xe} and a change of variable $r=2\pi i \tau / \mathcal{L}$, we deduce that
\begin{align*}
\begin{split}
I=\int\limits_{C}-\frac {8}{\pi}\frac{1}{\mathcal{L}} &
\frac{\Gamma\left(1/2 -2\pi i \tau/\mathcal{L}\right)}{\Gamma\left(1/2 +2\pi i \tau/\mathcal{L}\right)}
\left(\frac{\pi^2}{32}\right)^{2 \pi i \tau/\mathcal{L}}\left(2-2^{4\pi i \tau/\mathcal{L}} \right)\zeta_K \left( 1-\frac {4\pi i \tau}{\mathcal{L}} \right) \\
& \times \frac{\phi\left(\tau\right)}{W(X)}\sum_{\substack{ \varpi \equiv 1 \bmod {(1+i)^3}}}  w\left( \frac {N(\varpi)}X \right)  N(\varpi)^{-2\pi i \tau/\mathcal{L}} \dif \tau,
\end{split}
\end{align*}
  where $C$ stands for the horizontal line $\Im(\tau)=-\mathcal{L}a' /(2\pi)$. \newline

 We now deform $C$ to the path $C'=C_0\cup C_1\cup C_2, $
where
$\,C_0= \left\{ \tau: \Im(\tau)= 0 , |\Re(\tau)| \geq \mathcal{L}^\varepsilon \right\}\!, ~C_1=\\ \left\{ \tau: \Im(\tau) = 0 ,  \eta \leq |\Re(\tau)| \leq \mathcal{L}^\varepsilon \right\} ~\mbox{and}~
C_2= \left\{ \tau: |\tau| = \eta, \Im(\tau) \leq 0 \right\}\!,  $	
 for some small $\varepsilon$, $\eta>0$. \newline

The integration of $I$ over $C_0$ can be estimated trivially by making use of the rapid decay of $\phi$ to be of $O(\mathcal{L}^{-1})$. 
Using Taylor expansions, we see that on $C_1\cup C_2$,
\begin{align*}
  -\frac {8}{\pi}\frac{1}{\mathcal{L}}
\frac{\Gamma\left(1/2 -2\pi i \tau/\mathcal{L} \right)}{\Gamma\left(1/2 +2\pi i \tau/\mathcal{L} \right)}
\left(\frac{\pi^2}{32}\right)^{2 \pi i \tau/\mathcal{L}}\left(2-2^{4\pi i \tau/\mathcal{L}} \right)\zeta_K \left( 1-\frac {4\pi i \tau}{\mathcal{L}} \right)
=\frac 1{2\pi i \tau}+O\bigg( \frac{|\tau|+1}{\mathcal{L}}\bigg).
\end{align*}
  It follows from \eqref{Wpower} that the integrand of $I$ on $C_1\cup C_2$ equals
\begin{equation} \label{lastO}
\begin{split}
 \left( \frac 1{2\pi i \tau}+O\left( \frac{|\tau|+1}{\mathcal{L}}\right) \right) & \frac{\phi\left(\tau\right)}{W(X)}\sum_{\substack{ \varpi \equiv 1 \bmod {(1+i)^3}}}  w\left( \frac {N(\varpi)}X \right)  N(\varpi)^{-2\pi i \tau/\mathcal{L}} \\
=&  \left(\frac 1{2\pi i \tau}+O\left( \frac{|\tau|+1}{\mathcal{L}}\right) \right) \left( X^{-2\pi i \tau/\mathcal{L}}  +O\left( \frac {|\tau|^2}{\mathcal{L}}+\frac {1}{\mathcal{L}} \right) \right) \phi\left(\tau\right) \\
=&  \frac{\phi(\tau) e^{-2\pi i \tau}}{2\pi i \tau} +O\left( \frac {1+|\tau|+|\tau|^2}{\mathcal{L}|\tau|}\phi(\tau) \right).
\end{split}
\end{equation}
Now the integrals of the above main term and $O$-term over $C_1 \cup C_2$ are treated separately and independently.   In dealing with integral of the main term in \eqref{lastO}, we shall take $\eta$ to zero, while the integral of $O$-term in \eqref{lastO} is estimated directly for a fixed $\eta$.  This treatment is similar to that for the analogous expressions in \cite{FPS1}. \newline

To deal with the $O$-term in \eqref{lastO}, it suffices to consider the integral
\[  \int\limits_{C_1\cup C_2} \frac {\phi(\tau)}{\mathcal{L}|\tau|}  \dif \tau .   \]
Because of the rapid decay of $\phi$, we easily get that the above is $\ll \mathcal{L}^{-1}$, upon fixing a value of $\eta$.  Now for
\[  \int\limits_{C_1\cup C_2} \frac{\phi(\tau) e^{-2\pi i \tau}}{2 \pi i \tau} \dif \tau , \]
we take $\eta$ to zero and get that the above is,
\[ \frac {\phi(0)}2-\frac 12 \int\limits_{\mr} \frac{\sin(2\pi \tau)}{2\pi \tau} \phi(\tau) \dif\tau + O(\mathcal{L}^{-1}), \]
which absorbs the integral of the $O$-term in \eqref{lastO}.  \newline

The assertion of Theorem \ref{quadraticmainthm} now follows by combining \eqref{Dsimplified}, \eqref{equation lemma 4.1} with the above expression.

\medskip
\noindent
{\bf Acknowledgement.} Parts of this work were done when P.G. visited UNSW in September 2019. He wishes to thank UNSW for the invitation, financial support and warm hospitality during his pleasant stay. Finally, the authors thank the anonymous referee for his/her very careful reading of this manuscript and many helpful comments and suggestions.


\bibliographystyle{amsalpha}

\authoraddresses{
\textbf{Peng Gao}\\
School of Mathematical Sciences\\
Beihang University\\ 
Beijing 100191 China\\
\email penggao@buaa.edu.cn

\textbf{Liangyi Zhao}\\
School of Mathematics and Statistics\\
University of New South Wales\\
Sydney NSW 2052 Australia\\
\email l.zhao@unsw.edu.au
}

\end{document}